\documentclass[12pt]{article}
\usepackage{hyperref}
\usepackage[T1]{fontenc}

\usepackage[english]{babel}

\usepackage{tikz-cd}
\usepackage{geometry}

\usepackage{amsmath}
\usepackage{amsfonts}
\usepackage{amssymb}
\usepackage{amsthm}
\usepackage{graphicx}
\usepackage{dsfont}

\usepackage{tikz}

\usepackage{float}
\usetikzlibrary{calc,decorations.pathmorphing,shapes}
\usepackage{soul} 
\usepackage{url}

\usepackage{multicol}

\usepackage{xcolor}
\usepackage{menukeys}
\theoremstyle{definition}
\newtheorem{thm}{Theorem}[section]

\newtheorem{defn}{Definition}
\newtheorem{lemma}[thm]{Lemma}
\newtheorem{prop}[thm]{Proposition}

\newtheorem{cor}[thm]{Corollary}

\newtheorem{rmk}[thm]{Remark}
\newtheorem{question}{Question}
\newtheorem*{notation}{Notation}

\newtheorem*{combs}{Combs Lemma}
\newtheorem*{greedy}{The Greedy Lemma}
\newtheorem*{Counter-Examples}{Counter-Examples }
\newtheorem*{Pseudo-Power}{Pseudo-Power Theorem}

\newcommand{\interior}[1]{%
  {\kern0pt#1}^{\mathrm{o}}%
}
\newcommand{\Spine}[1]{\mathcal{S}(#1)}

\newcommand{\teeth}[1]{\mathcal{T}(#1)}

\DeclareMathOperator{\cf}{\mathrm{cf}}

\DeclareMathOperator{\Ord}{Ord}

\newcommand*{\axiomfont}[1]{\textsf{\textup{#1}}}

\newcommand{\ZFC}{\axiomfont{ZFC}}
\newcommand{\gch}{\axiomfont{GCH}}

\newcommand{\HC}{\axiomfont{HC}}
\newcommand{\SCH}{\axiomfont{SHC}}

\newtheorem*{thm*}{Pseudo-Power Theorem}
\makeatletter
\AtBeginEnvironment{thm*}{\def\@currentlabel{Pseudo-Power Theorem}}
\makeatother

\newtheorem*{thm**}{Counter-Examples }
\makeatletter
\AtBeginEnvironment{thm**}{\def\@currentlabel{Counter-Examples }}
\makeatother

\DeclareFontFamily{U}{rcjhbltx}{}
\DeclareFontShape{U}{rcjhbltx}{m}{n}{<->rcjhbltx}{}
\DeclareSymbolFont{hebrewletters}{U}{rcjhbltx}{m}{n}

\newtheorem*{theorem*}{Theorem}
\newtheorem*{conjecture*}{Conjecture}
\DeclareMathSymbol{\lamed}{\mathord}{hebrewletters}{108}
\DeclareMathSymbol{\mem}{\mathord}{hebrewletters}{109}
\DeclareMathSymbol{\ayin}{\mathord}{hebrewletters}{96}
\DeclareMathSymbol{\tsadi}{\mathord}{hebrewletters}{118}
\DeclareMathSymbol{\qof}{\mathord}{hebrewletters}{113}
\DeclareMathSymbol{\shin}{\mathord}{hebrewletters}{152}
\DeclareMathSymbol{\resh}{\mathord}{hebrewletters}{114}

\usepackage{hyperref}

\makeatletter
\newcommand{\subjclass}[2][1991]{%
  \let\@oldtitle\@title%
  \gdef\@title{\@oldtitle\footnotetext{#1 \emph{Mathematics subject classification.} #2}}%
}
\newcommand{\keywords}[1]{%
  \let\@@oldtitle\@title%
  \gdef\@title{\@@oldtitle\footnotetext{\emph{Key words and phrases.} #1.}}%
}
\makeatother
\author{Gabriel Fernandes}

\begin{document}



\title{Remarks on Halin's end degree Conjecture  }


\maketitle

\begin{abstract}
    We prove new instances of \emph{Halin's end degree conjecture} ($\HC$) in $\ZFC$. In particular, we show that there is a proper class of cardinals $\kappa$ for which Halin's conjecture holds, answering two questions posed by Geschke, Kurkofka, Melcher, and Pitz (2023). We also investigate the relationship between $\HC$ and the \emph{Singular Cardinal Hypothesis}, deriving consistency strength from failures of the former. Moreover, we verify that Halin's conjecture fails on finite intervals of successors of singular cardinals in Merimovich’s model, yielding a new independence result concerning $\HC$.
\end{abstract}
\section{Introduction}

Given an infinite graph $G=(V,E)$, a sequence $r=\{v_n : n \in \omega\} \subseteq V$ is called a \emph{ray} of $G$ if it is injective and satisfies $\{v_n,v_{n+1}\} \in E$ for all $n\in\omega$. If $r_0$ and $r_1$ are rays of $G$ and there exists a ray $r_2$ such that both $r_0 \cap r_2$ and $r_2 \cap r_1$ are infinite, then $r_0$ and $r_1$ are said to be \emph{related}. This is an equivalence relation on the set of rays of $G$, whose classes are called \emph{ends}. In this paper, we characterize when certain configurations of ends occur, clarifying the influence of cardinal arithmetic on structural properties of infinite graphs.

If $\varepsilon$ is an end of $G$, the \emph{degree} of $\varepsilon$, denoted  $\deg(\varepsilon)$, is the supremum of the cardinalities of sets $\mathcal{F}\subseteq\varepsilon$ consisting of pairwise disjoint rays. Halin proved in \cite{halin1965maximalzahl} that this supremum is always attained; see also \cite{diestel}\footnote{Ends and their degrees were introduced in \cite{halin1965maximalzahl}.}.

Halin's end degree conjecture (see Definition \ref{Def_HC}) says that an end of a graph can be characterized in terms of certain typical ray configurations, which would generalize his famous grid theorem. Recent progress in \cite{geschke2023halin} shows that cardinal arithmetic  influences structural properties of ends, revealing  interesting connections between set theory and infinite graph theory. Here, we advance this line of investigation by providing an almost complete characterization of when Halin’s end degree conjecture holds. Our main result describes, for each cardinal $\kappa > 2^{\omega}$,  whether Halin's end degree conjecture holds for ends of degree $\kappa$, showing that its validity depends solely on the position of $\kappa$ within the cardinal hierarchy defined by a cardinal function we call $\resh$. This  allows us to answer two questions from \cite{geschke2023halin}, derive new independence and consistency strength results concerning Halin's end degree conjecture.

We proceed with a few more definitions in order to precisely state Halin's end degree conjecture:
    \begin{itemize}
 
    \item \textbf{(Interior of a path)} If $c=\{p_1,\dots,p_m\}$ is a path in a graph $G$, then the \textit{interior} of $c$ is the set $\{p_2,\dots,p_{m-1}\}$, denoted by $\interior{c}$.
    
    \item \textbf{(Independent paths)} Let $\mathcal{R}$ be a collection of rays of a graph $G$, and let $c_0,c_1$ be paths in $G$. The paths $c_0$ and $c_1$ are \textit{$\mathcal{R}$-independent} if $\interior{c_0} \cap \bigcup \mathcal{R} = \emptyset = \interior{c_1} \cap \bigcup \mathcal{R}$ and $\interior{c_0} \cap \interior{c_1} = \emptyset$.
    
    \item \textbf{($( \mathcal{R}, \mathcal{P})$-associated graph)} Let $\varepsilon$ be an end of a graph $G$, and let $(\mathcal{R},\mathcal{P})$ be such that $\mathcal{R} \subseteq \varepsilon$ is a set of mutually disjoint rays and $\mathcal{P}$ is a set of $\mathcal{R}$-independent paths. The \textit{$(\mathcal{R},\mathcal{P})$-associated graph} $(\mathcal{V},\mathcal{E})$ is defined by $\mathcal{V}=\mathcal{R}$ and $\{r_0,r_1\} \in \mathcal{E}$ if and only if there are infinitely many mutually disjoint paths in $\mathcal{P}$ connecting $r_0$ and $r_1$.
    
    \item \textbf{(Ray graph)} Let $\varepsilon$ be an end of a graph $G$ and a $(\mathcal{R},\mathcal{P})$ be such that $\mathcal{R} \subseteq \varepsilon$ is a set of mutually disjoint rays, $|\mathcal{R}|=\deg(\kappa)$ and $\mathcal{P}$ is a set of $\mathcal{R}$-independent paths. When the $(\mathcal{R},\mathcal{P})$-associated graph $(\mathcal{V},\mathcal{E})$ is connected, we say that $(\mathcal{R},\mathcal{P})$ is a \textit{ray graph for $\varepsilon$}.
\end{itemize}

\begin{defn}[\textbf{Halin's End Degree Conjecture}]\label{Def_HC}
Let $\kappa$ be a regular cardinal. We say that \textit{Halin's end degree conjecture holds for $\kappa$} if, for every graph $G$ and every end $\varepsilon$ of $G$ with $\deg(\varepsilon)=\kappa$, there exists $(\mathcal{R},\mathcal{P})$ a ray graph for $\varepsilon$.

\end{defn}

\begin{notation}
When Halin's end degree conjecture holds for $\kappa$, we write that \textit{$\HC(\kappa)$ holds}.
\end{notation}


Halin conjectured in \cite{halin2000miscellaneous} that $\HC(\kappa)$ holds for all cardinals $\kappa$. Twenty years later, \cite{geschke2023halin} showed that whether it holds or not depends on the degree $\kappa$ of the ends.

In \cite{geschke2023halin}, it is shown that for any sufficiently large regular cardinal $\kappa$, it is consistent that $\HC(\kappa)$ fails. We approach the problem of determining when $\HC(\kappa)$ holds from a global perspective. Our main result, Theorem~\ref{main}, shows that whether $\HC(\kappa)$ holds depends essentially on cardinal arithmetic properties of $\kappa$.

We prove in Theorem~\ref{thm_strong} that there is a proper class of regular cardinals for which $\HC(\kappa)$ holds. This provides a positive answer to \cite[Questions~1) and~3)]{geschke2023halin}\footnote{For a consistency result regarding a negative answer to \cite[Question~2)]{geschke2023halin}, see \cite{aurichifernandesjunior}.}.

In Section~\ref{Section_Basics}, we prove several positive results concerning $\HC(\kappa)$, extending the analysis in \cite[Section~4]{geschke2023halin}.

In Section~\ref{Section_Resh_Thm}, we introduce the function $\resh$ to track certain $\omega$-inaccessible cardinals. Our main result, Theorem~\ref{main}, is proved in this section. It establishes that for a given cardinal $\kappa > 2^{\omega}$, it is possible to determine whether $\HC(\kappa)$ holds based on the interval $[\resh(\gamma),\resh(\gamma+1)]$ containing $\kappa$.

\begin{defn}[\cite{CardArthSkept}]
A cardinal $\kappa$ is said to be \emph{$\theta$-inaccessible} if and only if for all $\mu < \kappa$, we have $\mu^\theta < \kappa$.
\end{defn}

\begin{notation}
If $\mu$ and $\lambda$ are cardinals with $\mu \leq \lambda$, we set $(\mu, \lambda] := \{ \theta \leq \lambda \mid \theta \text{ is a cardinal and } \mu < \theta \leq \lambda \}$. Other interval notations are similar and will be clear from context.
\end{notation}

\begin{thm}\label{main}
Define $\resh: \mathrm{ON} \to \mathrm{CARD}$ by
\[
\resh(\beta)=
\begin{cases}
  2^{\omega}, & \text{if } \beta = 0; \\[4pt]
  (\resh(\gamma)^{+\omega})^{\omega}, & \text{if } \beta = \gamma + 1 \text{ and } \mathrm{cf}(\gamma) \neq \aleph_0; \\[4pt]
  \resh(\gamma)^{\omega}, & \text{if } \beta = \gamma + 1 \text{ and } \mathrm{cf}(\gamma) = \aleph_0; \\[4pt]
  \bigcup_{\alpha < \beta} \resh(\alpha), & \text{if } \beta \text{ is a limit ordinal.}
\end{cases}
\]

Suppose $\kappa > \mathfrak{c}$ is a cardinal such that whenever $\kappa = \resh(\gamma)$
for some limit ordinal $\gamma$ with $\mathrm{cf}(\gamma)$ not $\omega$-inaccessible, 
then necessarily $\mathrm{cf}(\gamma) > 2^{\omega}$. Under this hypothesis, the following hold:

\begin{itemize}
    \item[(a)] If $\kappa \in (\resh(\gamma),\, \resh(\gamma)^{+\omega}]$ and $\mathrm{cf}(\kappa) \neq \aleph_0$, then $\HC(\kappa)$ holds.

    \item[(b)] If $\kappa \in (\resh(\gamma)^{+\omega},\, \resh(\gamma+1)]$ and $\mathrm{cf}(\kappa) \neq \aleph_0$, then $\HC(\kappa)$ fails.

    \item[(c)] If $\kappa \in (\resh(\gamma)^{+\omega},\, \resh(\gamma+1)]$ and $\mathrm{cf}(\gamma)=\aleph_0$, then $\HC(\kappa)$ fails.

    \item[(d)] $\HC(\kappa)$ holds whenever $\kappa = \resh(\gamma)$ is a singular cardinal and either $\mathrm{cf}(\gamma)=\omega$ or $\mathrm{cf}(\gamma)$ is $\omega$-inaccessible.

    \item[(e)] $\HC(\kappa)$ fails whenever $\kappa = \resh(\gamma)$ is a singular cardinal and $\mathrm{cf}(\kappa)$ is not $\omega$-inaccessible.
\end{itemize}

\end{thm}

In particular, item (a) implies that $\HC$ holds on the interval $(\mathfrak{c}, \mathfrak{c}^{+\omega}]$ and fails on  $[\mathfrak{c}^{+\omega+1}, (\mathfrak{c}^{+\omega})^{\omega}]$.

 Section~\ref{Section_Merimovich}, explores connections between $\SCH$ and $\HC$, deriving large cardinal strength from failures of the latter, and verifies that $\HC$ fails on finite intervals in Merimovich’s models \cite{merimovichmodel}, where for each $n \in \omega\setminus\{0,1\}$ we have that $2^{\kappa}=\kappa^{+n}$ for every cardinal $\kappa$. 

\section{Positive cases of $\HC$ \label{Section_Basics}}

The results in this section are  based on the results on bipartite graphs of \cite[Section 4]{geschke2023halin}. By focusing on $\omega$-inaccessible cardinals we are able to eliminate the hypothesis that appears in \cite[Section 4]{geschke2023halin}, that $\gch$ holds. At the end of this section we use the techniques from \cite{geschke2023halin} and  apply our results on bipartite graphs in order to obtain positive instances of $HC(\kappa)$.

\begin{defn}
A \emph{$(\lambda,\kappa)$-graph} is a bipartite graph $G = (A, B, E)$, where
    \begin{itemize}
        \item $E$ is the edge-relation on the vertex set  $A \cup B$;
        \item  $|A| = \lambda < \kappa = |B|$; 
        \item  $\left(A \times \{b\}\right) \cap E$ is infinite, for every $b \in B$; and
        \item $A^2 \cap E = B^2 \cap E = \emptyset$.   
    \end{itemize}  
\end{defn}

\begin{defn}
   If $(A,B)$ is a $(\lambda,\kappa)$-graph, we say that $(A',B')$ is a subgraph of $(A,B)$ if and only if $A'\subseteq A$, $B' \subseteq B$ and for every $b \in B'$ the set $(A'\times\{b\}) \cap E$ is infinite. 
\end{defn}

\begin{notation}
    If $G=(V,E)$ is a graph and $v$ is a vertex of $G$, we set $N_{G}(v)=\{w \in V \mid \{v,w\}\in E\}$.
\end{notation}

\begin{prop} \label{Prop_HCsucc} Let $\mu$ be a cardinal $\geq 2$. Then for every natural $n \geq 1$ it holds that $(\mu^{\omega})^{+n}$ is a regular $\omega$-inaccessible cardinal and that $((\mu^{\omega})^{+n})^{\omega}=(\mu^{\omega})^{+n}$. Moreover  $(\mu^{\omega})^{+\omega})$ is $\omega$-inaccessible. 
\end{prop}
\begin{proof} 
We prove it by induction on $n \geq 1$. For $n=1$, let $\kappa=(\mu^{\omega})^{+}$. Since every successor cardinal is regular, it follows that $\kappa$ is regular. 

For every $\mu$, it holds that $(\mu^{\omega})^{\omega}=\mu^{\omega}$. If $\theta < \kappa$, then $\theta \leq \mu^{\omega}$ and therefore $\theta^{\omega} \leq (\mu^{\omega})^{\omega} = \mu^{\omega} < \kappa$. Thus $\kappa$ is $\omega$-inaccessible. 

If $f \in \kappa^{\omega}$, from the regularity of $\kappa > \mu^{\omega} \geq  2^{\omega} > \omega$, it follows that $sup(Im(f)) = \alpha < \kappa$ for some $\alpha$. Then $f \in \alpha^{\omega}$. 

Notice that $\alpha < \kappa$, implies $|\alpha|\leq \mu^{\omega}$, then $|\alpha|^{\omega} \leq \mu^{\omega} $. It follows that $\kappa^{\omega}=|\bigcup_{\alpha \in \kappa} \alpha^{\omega}|$. Then $\kappa^{\omega} \leq \kappa \times \mu^{\omega} = \kappa  $.

Suppose the proposition holds true for $n$ and let $\delta=(\mu^{\omega})^{+n+1}$. If $\theta < \delta$, then $\theta \leq (\mu^{\omega})^{+n}$. From our induction hypothesis $(\mu^{\omega})^{+n} =((\mu^{\omega})^{+n})^{\omega} \geq \theta^{\omega}$. Then $\delta$ is a regular $\omega$-inaccessible cardinal. 

Let $f \in \delta^{\omega}$. Since $\delta$ is regular, there is an ordinal $\beta \in \delta$ such that $f \in \beta^{\omega}$. Then $\delta^{\omega}=|\bigcup_{\beta \in \delta} \beta^{\omega}|$. Using the induction hypothesis that $((\mu^{\omega})^{+n})^{\omega}=(\mu^{\omega})^{+n}$, we have $\delta^{\omega} \leq \delta \times (\mu^{\omega})^{+n} = \delta$.  This concludes the induction.  

In order to verify that $(\mu^{\omega})^{+\omega}$ is an $\omega$-inaccessible cardinal, let
 $\delta < (\mu^{\omega})^{+\omega}$. For some $n \in \omega$, we have $\delta < (\mu^{\omega})^{+n}$. Hence $\delta^{\omega} \leq ((\mu^{\omega})^{+n})^{\omega}=(\mu^{\omega})^{+n} < (\mu^{\omega})^{+\omega}$.
\end{proof}

\begin{prop} \label{Prop_HCSucc2}
    Suppose that $\kappa$ is a singular cardinal and $\kappa$ is $\omega$-inaccessible. Then there is a sequence $\langle \theta_{\beta} \mid \beta < \cf(\kappa) \rangle$ increasing and cofinal in $\kappa$ such that each $\theta_{\beta}$ is a regular $\omega$-inaccessible cardinal.
\end{prop}
\begin{proof}
    Let $\langle \delta_{\beta} \mid \beta < \cf(\kappa) \rangle$ be a sequence of cardinals that is increasing and cofinal in $\kappa$. For each $\beta < \cf(\kappa)$ let $\mu_{\beta}=(\delta_{\beta}^{\omega})^{+}$. From Proposition \ref{Prop_HCsucc} each $(\delta_{\beta}^{\omega})^{+}$ is an $\omega$-inaccessible cardinal. Since $\kappa$ is $\omega$-inaccessible and singular, we have that $(\delta_{\beta}^{\omega})^{+} < \kappa$ for every $\beta < \cf(\kappa)$. We can then extract a  sequence $\langle \theta_{\beta} \mid \beta < \cf(\kappa) \rangle$ from $\langle(\delta_{\beta}^{\omega})^{+} \mid \beta < \cf(\kappa) \rangle$ that is increasing and cofinal in $\kappa$ and such that every $\theta_{\beta}$ is a regular $\omega$-inaccessible cardinal.
\end{proof}

\begin{lemma}\label{bipartite1} Let $(A,B,E)$ be a $(\lambda,\kappa)$-graph such that  $(A \times \{b\})\cap E $ is countable, for every $b \in B$.
    \begin{itemize} 
        \item[(1)]  If $\kappa$ is a regular  $\omega$-inaccessible cardinal, then there exists a $(\aleph_{0},\kappa)$-subgraph of $(A,B, E)$. 
        \item[(2)] If $\cf(\kappa) < \kappa$, $\kappa$ is an $\omega$-inaccessible cardinal such that  $\cf(\kappa)$ is an $\omega$-inaccessible cardinal and $\lambda < \cf(\kappa)$, then  there exists a $(\aleph_{0},\kappa)$-subgraph of $(A, B, E)$.
        \item[(3)] \cite[Lemma 4.1]{geschke2023halin} If $\cf(\lambda) > \omega$ and $\cf(\kappa) \neq \cf(\lambda)$, then there is a  $(\lambda',\kappa)$-subgraph of $(A,B,E)$, for some $\lambda' < \lambda$.
     
        \item[(4)] If $\kappa$ is a singular cardinal that is also an $\omega$-inaccessible cardinal and $cf(\kappa)$ is an $\omega$-inaccessible cardinal, then either there exists an $(\aleph_{0},\kappa)$-subgraph of $(A,B, E)$; or else there exists a sequence $\langle \theta_{\beta} \mid \beta <\cf(\kappa) \rangle$ of regular $\omega$-inaccessible cardinals cofinal in $\kappa$ and a sequence $\langle (A_{\beta}, B_{\beta})  \mid \beta <\cf(\kappa) \rangle $ satisfying the following two conditions: 
\begin{itemize}
    \item[(i)] for each $\beta < \cf(\kappa)$ we have that $(A_{\beta},B_{\beta})$ is a $(\aleph_{0},\theta_{\beta})$-subgraph of $(A,B)$ 
    \item[(ii)] for  $\beta < \beta' < \cf(\kappa)$ we have $A_{\beta} \cap A_{\beta'}=\emptyset$ and $B_{\beta} \cap B_{\beta'} = \emptyset$. 
    
\end{itemize} 
     \end{itemize}
\end{lemma}

\begin{proof} (1)  Since $\kappa$ is $\omega$-inaccessible and $\lambda < \kappa$, we have $\lambda^{\omega} < \kappa$. For each $b \in B$ let $A_{b} = dom((A \times\{b\}) \cap E)$, i.e. $A_{b}$ is such that $A_{b}\times\{b\} = (A\times\{b\})\cap E$. Let $f:B \rightarrow [A]^{\omega}$ be defined by $f(b)=A_{b}$ for each $b \in B$. Then $f$ is such that $|B|=|dom(f)|=\kappa$ and $|ran(f)| \leq |[A]^{\omega}|=\lambda^{\omega} < \kappa$. 
By the regularity of $\kappa$, it follows that there exist $A' \subseteq A$ such that $|f^{-1}[\{A'\}]|=\kappa$. Let $B' = f^{-1}[\{A'\}]$.  Thus $(A',B')$ is a $(\aleph_{0},\kappa)$-subgraph of $(A,B)$.

(2)    From $\lambda < \cf(\kappa)$, it follows that $\lambda^{\omega} < \cf(\kappa)$. Analogous to item (1), define a function $g:B \rightarrow [A]^{\omega}$ by setting $g(b)=dom((A \times \{b\}) \cap E)$ for each $b \in B$. Since $|[A]^{\omega}| = \lambda^{\omega} < cf(\kappa)$, there is $A' \in [A]^{\omega}$ such that $|g^{-1}[A']|=\kappa$. Let $B'= g^{-1}[A']$. Hence $(A',B')$ is a $(\aleph_0,\kappa)$-subgraph of $(A,B)$.

(4) Suppose that there is no $(\aleph_{0},\kappa)$-subgraph of $(A,B)$. Let $\langle \theta_{\beta} \mid \beta \in cf(\kappa) \rangle$ be a cofinal sequence in $\kappa$ given by Propostition \ref{Prop_HCSucc2} such that for all $\beta \in cf(\kappa)$ it holds that $\theta_{\beta}$ is a regular $\omega$-inaccessible cardinal. We will define a sequence $\langle (A_{\beta}, B_{\beta})  \mid \beta <\cf(\kappa) \rangle $ satisfying the following two conditions: 
\begin{itemize}
    \item[(i)] for each $\beta < \cf(\kappa)$ we have that $(A_{\beta},B_{\beta})$ is a $(\aleph_{0},\theta_{\beta})$-subgraph of $(A,B)$ 
    \item[(ii)] for  $\beta < \beta' < \cf(\kappa)$ we have $A_{\beta} \cap A_{\beta'}=\emptyset$ and $B_{\beta} \cap B_{\beta'} = \emptyset$. 
    
\end{itemize} Suppose we have built a sequence $\langle (A_{\alpha},B_{\alpha}) \mid \alpha < \beta \rangle $ satisfying (i) and (ii) for $\alpha \in \beta$.  Let $A'_{\beta}= A \setminus \bigcup_{\alpha< \beta} A_{\alpha}$ and  

\[B'_{\beta} := \left\{b \in B \setminus \bigcup_{\alpha < \beta} B_{\alpha} \mid dom((A \times \{b\}) \cap E) \cap A'_{\beta} \text{ is infinite}\right\}.\]

Since each $A_{\alpha}$ is countable, it follows that  $|\bigcup  \{A_{\alpha} \mid \alpha < \beta\} | < \cf(\kappa)$. If $|B'_{\beta}|<\kappa$, then  $|B\setminus B'_{\beta}|=\kappa$. Notice that $$B \setminus B_{\beta}^{'}=\{b \in B \mid b \in \bigcup_{\alpha < \beta}B_{\alpha} \vee (b \in B\setminus \bigcup_{\alpha < \beta}B_{\beta}  \wedge dom((A\times\{b\})\cap E) \cap A_{\beta}' \text{ is finite } )\}$$  and from our induction hypothesis, it follows that $|\bigcup_{\alpha < \beta} B_{\beta}| \leq sup\{ \theta_{\alpha} \mid \alpha < \beta \} < \kappa$.
Therefore $$|\{b \in B \setminus \bigcup_{\alpha < \beta}B_{\beta} \mid dom((A\times\{b\})\cap E) \cap A'_{\beta} \text{ is finite }\}|=\kappa$$ and since $A'_{\beta}=A \setminus \bigcup_{\alpha < \beta}A_{\alpha}$, it follows that 
$$|\{b \in B \setminus \bigcup_{\alpha < \beta}B_{\beta} \mid dom((A\times\{b\})\cap E) \cap \bigcup_{\alpha < \beta} A_{\alpha} \text{ is infinite }\}|=\kappa$$

Applying $(2)$ of Lemma \ref{bipartite1} we find a $(\aleph_0,\kappa)$-subgraph of   
$\left(\bigcup_{\alpha < \beta} A_{\alpha}, B \setminus B'_{\beta}\right)$, and hence a $(\aleph_0,\kappa)$-subgraph of $(A,B)$, which contradicts our hypothesis on $(A,B)$. Thus $|B'_{\kappa}|=\kappa$. Let $B''_{\beta} \subset B'_{\beta}$ such that   $|B''_{\beta}|=\theta_{\beta}$ and  $A''_{\beta}\subseteq A'_{\beta}$ such that $A''_{\beta}=\{ a \in A'_{\beta} \mid \exists b ( b \in B''_{\beta} \wedge \{a,b\} \in E \}$. Then by item (1) of Lemma \ref{bipartite1} there is $(A'''_{\beta},B_{\beta}^{'''})$ a $(\aleph_0,\theta_{\beta})$ sub-graph of $(A''_{\beta},B''_{\beta})$. It follows that $(A'''_{\beta},B'''_{\beta})$ is a $(\aleph_0,\theta_{\beta})$-subgraph  of $(A,B)$ satisfying (i) and (ii) as desired and we set $(A_{\beta},B_{\beta})$ to be $(A'''_{\beta},B'''_{\beta})$. \qedhere
\end{proof}

Let us introduce some more notation in order to state two results from \cite{geschke2023halin} that we will need: 

\begin{defn}[Combs]
Let $\varepsilon$ be an end of a graph $G=(V,E)$ and $U \subseteq V$ disjoint from $\varepsilon$. An  $(\varepsilon, U)$-comb is a set of the form $C=R \cup \bigcup\mathcal{D}$, where $R$ is a ray that belongs to $\varepsilon$ and $\mathcal{D}$ is a collection of independent paths with ending points in $U$. We say that $v \in V$ is a tooth of a $(\varepsilon,U)$-comb $C$ if $v \in C \cap U$. We set $\teeth{C}=\{v \in V \mid v \text{ is a tooth of } C\}$ and $\Spine{C}=R$.
    
\end{defn}
Since Halin's conjecture can be phrased in terms of stars and we shall exploit that fact very often, we introduce some notation on stars and ray stars. 
\begin{defn}[Stars]
Let $\kappa$ be a cardinal. A \textit{$\kappa$-star} is a graph of the form $G=(\{v_{\alpha} \mid \alpha < \kappa\}, \{ \{v_0,v_{\alpha} \} \mid 0 < \alpha < \kappa \})$ such that $\beta < \gamma < \kappa $ implies $v_{\beta} \neq v_{\gamma}$. A \textit{$\kappa$-star of rays} is a triple $S=(G,\{r_{\alpha} \mid \alpha < \kappa\},\{p_{\alpha,j} \mid \alpha < \kappa \wedge j \in \omega \})$ consisting of the following:\begin{itemize}
    \item[(i)] A graph $G=(V,E) $ with $V=\bigcup_{\alpha<\kappa}r_{\alpha} \cup \bigcup\{ p_{\alpha,j} \mid \alpha < \kappa \wedge j \in \omega \}$.
    \item[(ii)] $\{ r_{\alpha} \mid  \alpha < \kappa \}$ is a set of mutually disjoint rays of $G$.   
    \item[(iii)]  $\{p^{\alpha}_{n} \mid n \in \omega \wedge \alpha \in \kappa\} $ is a family of $\{r_\alpha \mid \alpha < \kappa\}$-independent paths such that for each $\alpha \in \kappa$ and each $n \in \omega$ the path $p^{\alpha}_{n}$ connects $r_0$ to $r_{\alpha}$.
   \item[(iv)] Given $\alpha \in \kappa$, $\{p^{\alpha}_{n} \mid n \in \omega\}$ is a set of mutually disjoint paths.
\end{itemize} We will call $r_{0}$ the \textit{center} of the $\kappa$-star and for $\alpha > 0, ~ r_{\alpha}$  will be called a \textit{leaf ray} of $S$. 
\end{defn}

We need the following results from \cite{geschke2023halin}:

 \begin{combs}
     \hypertarget{Combs}{} \cite[Lemma 2.1]{geschke2023halin}  Let $\varepsilon$ be an end of a graph $G$ and $\mathcal{R}$ be a $\kappa$-sized collection of disjoint rays belonging to $\varepsilon$. If $U$ is a countable set of vertices, and $\mathcal{C}$ is an uncountable collection of internally disjoint $(\varepsilon$, $U)$-combs of $G$, then $\varepsilon$ contains a $|\mathcal{C}|$-star of rays whose leaf rays are the spines of (a subset of) combs in $\mathcal{C}$. 

 \end{combs}
 \begin{greedy}
 \hypertarget{The Greedy Lemma}{} \cite[Lemma 2.2] {geschke2023halin}  Let $\varepsilon$ be an end of a graph $G$ and $\mathcal{R}$ be a $\kappa$-sized collection of disjoint rays belonging $\varepsilon$. If $\mathrm{cf}(\kappa)>\aleph_{1}$, then there exists a set of vertices $U$ of $G$, and a $\kappa$-sized collection $\mathcal{C}$ of internally disjoint $(\varepsilon$, $U)$-combs of $G$, such that $|U|<\kappa$ and all spines of $\mathcal{C}$ are in $\mathcal{R}$. 
     
 \end{greedy}

We are ready to prove Theorem \ref{thm_strong}: 
\begin{thm} \label{thm_strong} Let $G$ be a graph with an end $\varepsilon$ such that $\deg(\varepsilon)=\kappa$, a regular $\omega$-inaccessible cardinal. For every $\mathcal{R}\subseteq \varepsilon$, family of mutually disjoint rays, such that $|\mathcal{R}|=\kappa$, there is a pair $(\mathcal{R}_{0},\mathcal{P}_{0})$, such that \begin{itemize}
    \item $\mathcal{R}_{0} \subseteq \mathcal{R}$.
    \item $|\mathcal{R}_{0}|=\kappa$.
    \item $\mathcal{P}_{0}$ is a family of $\mathcal{R}_{0}$-independent paths.
    \item $(\mathcal{R}_{0},\mathcal{P}_0)$ is a $\kappa$-star of rays.
\end{itemize}
In particular, $\HC(\kappa)$ holds for all $\omega$-inaccessible regular cardinals $\kappa$.
\end{thm} 
\begin{proof}      Let $G=(V,E)$ be a graph with an end $\varepsilon$ of degree $\kappa$. Let $\mathcal{R} \subseteq \varepsilon$ be a family of pairwise disjoint rays such that $|\mathcal{R}|=\kappa$. We will find $(\mathcal{R}_{0},\mathcal{P}_{0})$ a ray graph with for $\varepsilon$ with $\mathcal{R}_{0} \subseteq \mathcal{R}$. 

 Since $\kappa  > \omega_{1}$ and $\mathrm{cf}(\kappa)=\kappa$, we can apply the \hyperlink{The Greedy Lemma}{Greedy Lemma} to $\mathcal{R}$ and find $U \subseteq V$ and $\mathcal{C}$ such that $|U| = \lambda$ for some $\lambda < \kappa$, $|\mathcal{C}|=\kappa$ and $\mathcal{C}$ is a collection of internally disjoint $(\varepsilon, U)$-combs in $G$ with all spines in $\mathcal{R}$.

We have two cases to consider: $|U|=\aleph_0$ and $|U|>\aleph_0$. If $|U|=\aleph_{0}$, then, we can apply   \hyperlink{Combs}{Combs Lemma} to $U$ and $\mathcal{C}$, and obtain a $\kappa$-star of rays whose leaf rays are spines of combs in $\mathcal{C}$. Since a $\kappa$-star of rays is a  ray graph of cardinality $\kappa$ we have found the $(\mathcal{R}_{0},\mathcal{P}_{0})$ we sought. 

Suppose $|U| > \aleph_{0}$. Consider the  $(\lambda,\kappa)$-graph $(U, \mathcal{C}, E)$ defined as follows:  $(v, C) \in E$ if, and only if, $v \in \teeth{C}$.   If $|U| > \aleph_{0}$, setting $\lambda= |U|$, using that $\kappa$ is a regular cardinal that is $\omega$-inaccessible, it follows by (1) of Lemma \ref{bipartite1} applied to the bipartite graph $(U,\mathcal{C},E)$ that there is  $(U',\mathcal{C}',E')$ a $(\aleph_0,\kappa)$-subgraph of $(U,\mathcal{C},E)$. It follows that $U'$ is a countable set of vertices and $\mathcal{C}'$ is a set of pairwise internally disjoint $(\varepsilon,U)$-combs.

For each comb $C \in \mathcal{C}'$ let $P_{C}$ be the comb obtained by removing the paths $p$ from $C$ that have as end points some element that is not in $U'$. Let $\mathcal{C}''=\{P_{C} \mid C \in \mathcal{C}'\}$. Then, we can apply \hyperlink{Combs}{Combs Lemma} to $U'$ and $\mathcal{C}''$ and obtain a $\kappa$-star of rays whose leaf rays are spines of combs of $\mathcal{C}''$. We found the $(\mathcal{R}_{0},\mathcal{S}_{0})$ that we sought. 
\end{proof}

Theorem \ref{thm_strong} together with Proposition \ref{Prop_HCsucc} provide a positive answer for  \cite[Question 1)]{geschke2023halin}: Is Halin's conjecture true for any $\kappa > \aleph_{\omega_{1}}$ such that $HC(\kappa)$?

\begin{cor}\label{Corollary_functiondegree}
    For every cardinal $\kappa$ there is $f(\kappa) \geq \kappa$ such that   every end $\varepsilon$ of degree $f(\kappa)$ contains a connected ray graph of size $\kappa$
\end{cor}
\begin{proof}
    Given a cardinal $\kappa$, if $\kappa = (\mu^{\omega})^{+}$ for some cardinal $\mu$, then by Proposition \ref{Prop_HCsucc} and Theorem \ref{thm_strong}, we can set $f(\kappa)=\kappa$. If $\kappa$ is not of the form $(\mu^{\omega})^{+}$, let $f(\kappa) = (\kappa^{\omega})^{+}$. Again, by Proposition \ref{Prop_HCsucc} and Theorem \ref{thm_strong}, for every end $\varepsilon$ with $\deg(\varepsilon)=f(\kappa)$ we can find a connected ray graph of cardinality $f(\kappa)$ contained in $\varepsilon$. 
\end{proof}

Corollary \ref{Corollary_functiondegree} provides a positive answer to the following question: \cite[Question 3) ]{geschke2023halin} Is it true that for every cardinal $\kappa$ there is $f(\kappa) \geq \kappa$, such that for every end $\varepsilon$ of degree $f(\kappa)$ contains a  ray graph of size $\kappa$ ?

The following lemma will be used in the singular cardinals case in the proof of Theorem \ref{main}.  

\begin{thm}\label{LemmaHCsingular} Suppose that $\kappa$ is an $\omega$-inaccessible cardinal and that $cf(\kappa)=\omega$ or $cf(\kappa)$ is an $\omega$-inaccessible cardinal. Then $\HC(\kappa)$ holds.
\end{thm}

\begin{proof} Given a graph $G=(V,E)$ with an end $\varepsilon$ such that $\mathrm{deg}(\varepsilon)=\kappa$, let $\mathcal{F} \subseteq \varepsilon$ be a set of pairwise disjoint rays such that $|\mathcal{F}|=\kappa$.

We consider two cases: $\cf(\kappa)=\omega$ and $\cf(\kappa) > \omega_{1}$. We start with the case $\cf(\kappa)=\omega$. Using that $\kappa$ is $\omega$-inaccessible we can fix $\langle \mu_{n} \mid n < \omega  \rangle$ an increasing sequence of cardinals that is cofinal  in $\kappa$ such that for each $n <\omega$ it holds that $(\mu_{n})^{\omega}=\mu_{n}$. By Proposition \ref{Prop_HCsucc} and Theorem \ref{thm_strong}, for each $n < \omega = \cf(\kappa)$, it is possible to find $S_n=(G_n,\{r_{\alpha}^{n} \mid \alpha < \mu_{n}^{+}\},\{p^{n}_{\alpha,j} \mid \alpha < \mu_{n}^{+} \wedge k \in \omega \})$ a $\mu_{n}^{+}$-star of rays such that $\{r_{\alpha}^{n} \mid \alpha < \mu_{n}^{+}\} \subseteq \mathcal{F}$ is a set of mutually disjoint rays and a set $\mathcal{P}_{n}=\{p^{n}_{\alpha,j} \mid \alpha < \mu_{n}^{+} \wedge k \in \omega \}$ a set of  $\{r_{\alpha}^{n} \mid \alpha < \mu_{n}^{+}\}$-independent paths such that $((\bigcup S_{n}, E \cap \left[\bigcup S_n\right]^{2}), S_{n},\mathcal{P}_{n})$ is a $\mu_{n}^{+}$-star of rays with center $r_{0}^{n}$. Notice that  for each $\alpha$ and each $n \in \omega$ the set $C_{\alpha}^{n}=r_{\alpha}^{n}\cup\{p^{n}_{\alpha,j} \mid j \in \omega\}$ is a comb, and $\{C^{n}_{\alpha} \mid \alpha < \kappa \}=\mathcal{C}$ is a set of pairwise internally disjoint $(\varepsilon, \{r^{n}_{\alpha} \mid \alpha < \mu^{+}_{n} \}) $-combs. 

Since $|\bigcup_{k < n } S_{k} \cup \bigcup_{k < \omega} R_{k}|\leq \max\{ \mu^{+}_{k} \mid k < n \} = \mu^{+}_{n-1}$, there is a set $\mathcal{D}_{n} =\{C^{n}_{\alpha_\delta}\mid \delta \in \mu_{n}^{+}\}\subseteq \mathcal{C}$  of cardinality $\mu_{n}^{+}$ such that  for every $\delta \in \mu_{n}^{+}$ the interior of  $C_{\alpha_{\delta}}^{n}$ is disjoint from $\bigcup_{k < n} S_{k} \cup \bigcup_{k < \omega} R_{k}$. 

Hence, $\bigcup_{n \in \omega} \mathcal{D}_{n}$ is a set of  $(\varepsilon,\bigcup_{k \in \omega} R_{k})$-combs of cardinality $\kappa$ which elements are pairwise internally disjoint. Applying the  \hyperlink{Combs}{Combs Lemma} we obtain a $\kappa$-star of rays in $G$.

Next, suppose $\cf(\kappa) \neq \omega$. Then $\omega_1 \leq \omega^{\omega} < \cf(\kappa)$. Let $\varepsilon$ be an end of $G$ with $\deg(\varepsilon) =  \kappa$, by the \hyperlink{The Greedy Lemma}{ Greedy Lemma}, there are some set of vertices $U \subset V(G)$ with $|U| < \kappa$ and a $\kappa$-sized family $\mathcal{C}$ of internally disjoint $(\varepsilon,U)$-combs. 

Consider $H=(U,\mathcal{C},E)$, the bipartite $(|U|,\kappa)$-graph defined as follows: $\{z,C\}\in E$ if and only if $z \in \mathcal{T}(C)$ and $C \in \mathcal{C}$. By item (4) of Lemma \ref{bipartite1}, $H$ contains either an $(\aleph_0,\kappa)$-subgraph or a collection of disjoint $(\aleph_0,\theta_{\beta})$-subgraphs for $\{\theta_{\beta} \mid \beta < \cf(\kappa)\}$ cofinal in $\kappa$ with all $\theta_{\beta} > \max \{\aleph_{\omega+1}, \cf(\kappa)\}$ regular and $\omega$-inaccessible. 

If $H$ contains an $(\aleph_0,\kappa)$-subgraph, $(A,B)$, then every comb $C \in B$ is such that $|\mathcal{T}(C)\cap A|=\aleph_0$. For each comb $C = S \cup \bigcup D$ that is in $B$ let $C'$ be the comb obtained by deleting the paths $p \in D$ such that $p \cap A = \emptyset$. Then we can apply \hyperlink{Combs}{Combs Lemma} to $A$ and $\{C' \mid C \in B \}$ and obtain a $\kappa$-star of rays. This gives a  ray graph  as we sought.

Next, suppose we are in the case that $H$ does not contain an $(\aleph_0,\kappa)$-subgraph. So there is a collection $\{(A_{\beta},B_{\beta}) \mid \beta < \cf(\kappa) \}$ of disjoint $(\aleph_0,\theta_{\beta})$-subgraphs for $\{\theta_{\beta} \mid \beta < \cf(\kappa)\}$  a cofinal sequence in $\kappa$ such that all $\theta_{\beta}$ are greater than $ \max \{\aleph_{\omega+1}, \cf(\kappa)\}$, regular and $\omega$-inaccessible. 

Let $H_{\beta}=(A_{\beta},B_{\beta},E)$ be a $(\aleph_0, \theta_{\beta})$-subgraph of $(U,\mathcal{C},E)$. Similar to what we did in the previous case, for each $C \in B_{\beta}$, $C=S \cup \bigcup D$, let $C'$ be the comb obtained by deleting the paths $p \in D$ such that $p \cap A_{\beta}=\emptyset$. We can then apply \hyperlink{Combs}{Combs Lemma} to $A_{\beta}$ and $\{C' \mid C \in B_{\beta}\}$ and find a $\theta_{\beta}$-star of rays $S_{\beta}$. 

For each $\beta < \cf(\kappa)$, let $S_{\beta}$ be a $\theta_{\beta}$-star of rays obtained as above. Let $r_{0}^{\beta}$ be the center of $S_{\beta}$ and $r_{\alpha}^{\beta}$ be its leaf rays for $0 < \alpha < \theta_{\beta}$ and $\beta <\cf(\kappa)$. By Theorem \ref{thm_strong}, we have that $\HC(\cf(\kappa))$ holds, and we can obtain a $cf(\kappa)$-star of rays $S$ of size $\cf(\kappa)$ with leaf rays in $\{ r_{0}^{\beta} \mid \beta < \cf(\kappa)\}$.  Since $|S| < |S_{\beta}|$  for all $\beta < \cf(\kappa)$, shrinking $S_{\beta}$ if necessary, we may assume that each $S_{\beta}\cap S = \{r_0^{\beta}\}$ for each $\beta < \cf(\kappa)$. Then $S \cup \bigcup \{S_{\beta} \mid \beta < \cf(\kappa) \wedge r_0^{\beta} \in S\}$ yields a connected ray graph of size $\kappa$.  
\qedhere
\end{proof}

\section{$\HC$ on an interval-partition of the cardinals \label{Section_Resh_Thm}}

In this section, we prove a negative result regarding $\HC(\kappa)$ and use a cardinal function $\resh$ determine when $\HC$ holds or fails. The function $\resh$ helps tracking certain $\omega$-inaccessible cardinals. 

\begin{lemma}\label{Lemma_nosubgraph}
    Suppose  $\lambda$ and $\kappa$ are cardinals,  $\lambda$ is $\omega$-inaccessible and $\lambda < \kappa \leq \lambda^{\omega}$. There is a $(\lambda,\kappa)$-graph $(T,X,E)$ such that: $T\subseteq \lambda^{<\omega}$, $X \subseteq \lambda^{\omega}$,   $(T\cup X,\subseteq)$ is a tree and for all regular cardinals $\kappa'$ with $\lambda < \kappa' \leq \kappa$ the $(\lambda,\kappa)$-graph $(T,X,E)$ has no $(\aleph_{0},\kappa')$-subgraph.
\end{lemma}
\begin{proof} Let $X \subseteq \lambda^{\omega}$ such that $|X|=\kappa$ and let $$T=\{ s \in \lambda^{<\omega} \mid \exists n \exists f ( n \in \omega \wedge f \in X \wedge s=f\restriction n )\}$$
We have that $T \subseteq \lambda^{<\omega}$, then $|T| \leq \lambda$. Since $\lambda$ is $\omega$-strong, if $|T| < \lambda$, then the set $$[T]=\{f \in \lambda^{\omega} \mid \forall n ( n \in \omega \rightarrow f\restriction n \in T) \}$$ would have cardinality $< \lambda$ contradicting the fact that $X \subseteq [T]$  and $\lambda < \kappa = |X|$. Then $|T|=\lambda$. 

Let $(T,X,E)$ be a bipartite graph where $\{s,f\} \in E$ if and only if there is $n \in \omega$ such that $s=f\restriction n$. It follows that $(T,X,E)$ is a $(\lambda,\kappa)$-graph. 

Suppose that $A \subseteq T$ and $|A|=\aleph_{0}$. If $(A,B)$ is a subgraph of $(T,X)$ then $B \subseteq \{f \in \lambda^{\omega} \mid (\{ n \in \omega \mid f\restriction n \in A\} \text{ is infinite }) \}$. From the hypothesis that $\lambda$ is $\omega$-strong it follows that the set $[A]=\{f \in \lambda^{\omega} \mid (\{ n \in \omega \mid f\restriction n \in A\} \text{ is infinite }) \}$ has cardinality less or equal to $|A|^{\omega}| $, which is strictly smaller than $\lambda$. Therefore, for every $\kappa'$ such that $\lambda < \kappa' \leq \kappa$ it follows that  $(A,B)$ is not a $(\aleph_0,\kappa')$-subgraph of $(T,X)$. Hence $(T,X,E)$ is the $(\lambda,\kappa)$-graph we sought.  \qedhere

\end{proof}

The counter examples to $\HC(\kappa)$ that we construct will follow the methods developed in \cite{geschke2023halin}. These counterexamples will be obtained as inflations of  $T$-graphs, where $T$ is a tree. To this end, we introduce the notions of $T$-graphs and inflations of $T$-graphs below:
\begin{defn} \label{def_Tgraph}
    A tree $T$ is \emph{normal} in a graph $G$, if $V(G) = T$
and the two end-vertices of any edge of $G$ are comparable in~$T$. We call $G$ a
\emph{$T$-graph} if $T$ is normal in $G$ and the set of lower neighbors of
any point $t$ is $<_{T}$-cofinal in $\{t' <_{T} t \mid t' \in T \}$. We say that $G$  is a \emph{sparse $T$-graph} if for every limit $t \in T$ the set $\{t' <_{T} t \mid t' \in T \} \cap N(t)$ has order-type $\omega$. 
\end{defn}
    \begin{defn}
\label{def_rayinflation}
Let $G$ be a sparse $T$-graph for an order tree $T$ of height at most $\omega_1$. The \emph{ray-inflation} $G \sharp \mathbb{N}$ of $G$ is the graph with vertex set $T \times \mathbb{N}$, and the following edges: 
\begin{enumerate}
	\item For every $t \in T$ we add all the edges $(t,n)(t,n+1)$ with $n\in\mathbb{N}$ so that $R_t:=G[\,\{t\}\times\mathbb{N}\,]$ is a \emph{horizontal ray}.
	\item If $t \in T$ is a successor with predecessor $t'$ in the tree order $<_{T}$, we add all edges $(t,n)(t',n)$ for all $n \in \mathbb{N}$. 
	\item If $t \in T$ is a limit with down-neighbors $t_0 <_T t_1 <_T t_2 <_T \ldots$ in $G$ we add the edges $(t,n)(t_n,n)$ for all $n \in \mathbb{N}$.
\end{enumerate}
\end{defn}

The following lemma about ends of inflations will be very used to construct counter examples to Halin's conjecture: 

\begin{lemma}
\label{lem_ray_inflation_endstructure}
If $G$ is a sparse $T$-graph for $T$ an order tree of height at most $\omega_1$, then all the pairwise disjoint horizontal rays $R_t$ in the ray inflation $H=G \sharp \mathbb{N}$ belong to the same sole end $\varepsilon$ of $G \sharp \mathbb{N}$; in particular, $\deg(\varepsilon)=|T|$.
\end{lemma}

We include the following lemma which is a weaker version of \cite[Lemma 3.1]{geschke2023halin} formulated in a way that is more convenient for our purposes.

\begin{lemma} \label{lemma_stars}
    Let $\varepsilon$ be an end of a graph $G$. Let $\lambda < \kappa$ be infinite cardinals. Suppose that there is a pair $(\mathcal{R},\mathcal{P})$ such that
    \begin{itemize}
             \item[(a)] $\mathcal{R}$ is a set of mutually disjoint rays of $G$ such that $\mathcal{R}\subseteq \varepsilon$.
    \item[(b)] $\mathcal{P}$ is a set of $\mathcal{R}$-independent paths of $G$.
    \item[(c)] $|\mathcal{R}|=\kappa$. 
    \item[(d)] The $(\mathcal{R},\mathcal{P})$-associated $(\mathcal{V},\mathcal{E})$  is connected.
    \end{itemize}
    Then for every regular cardinal $\kappa'$ such that $\lambda < \kappa' \leq \kappa$ there is a $\kappa'$-start of rays in $\varepsilon$.
\end{lemma}
\begin{proof}
    Since $(\mathcal{V},\mathcal{E})$ is connected we can fix a vertex $v_{0} \in \mathcal{V}$ and write $\mathcal{V}=\bigcup_{n \in \omega} U_{n}$ where $U_{0}=\{v_{0}\}$ and $U_{n+1}=\{w \in \mathcal{V} \mid \exists s ( s \in U_{n} \wedge \{w,s\} \in \mathcal{E} \}$ for $n \in \omega$.

    Let $\kappa' \leq \kappa$ be a regular cardinal. We will find $v \in \mathcal{V}$ such that $|\{ w \in \mathcal{V} \mid \{w,v\} \in \mathcal{E}\}| \geq \kappa'$.
    Suppose otherwise, we will reach a contradiction showing that $|\mathcal{V}| < \kappa'$.  
    We prove by induction that for every $n \in \omega$ we have $|U_{n}| < \kappa'$. For $n=0$ we have $|\{v_0\}|=|U_{0}|=1$. Suppose that for a given $n \in \omega$ it holds that $|U_{n}| < \kappa'$. We have that $U_{n+1}=\bigcup_{v \in U_{n}} \{ w \in \mathcal{V} \mid \{w,v\} \in \mathcal{E})\}$, therefore $|U_{n+1}| < \kappa'$ and this concludes the induction. Then $|\mathcal{V}|< \kappa'$, which is a contradiction. Hence there is $v \in \mathcal{V}$ such that $|\{ w \in \mathcal{V} \mid \{w,v\} \in \mathcal{E}\}|\geq \kappa'$.

    Let $\theta = |\{ w \in \mathcal{V} \mid \{w,v\} \in \mathcal{E}\}|$, then it forms a $\theta$-star of rays $S$ and we can find a subset of $S'\subseteq S$ which is a $\kappa'$-star of rays. 
    \end{proof}
\begin{lemma} \label{lemma_counterexample} 
If $\lambda$ and $\kappa$ are cardinals, $\lambda$ is an $\omega$-inaccessible cardinal with $\cf(\lambda)=\omega$ and $\lambda < \kappa \leq \lambda^{\omega}$, then there is a graph $G$ with an end $\varepsilon$ such that $\deg(\varepsilon)=\kappa$ and for all regular cardinals $\kappa'$ such that $\lambda < \kappa' \leq \kappa$ there is no pair $(\mathcal{R},\mathcal{P})$ such that
\begin{itemize}
    \item[(a)] $\mathcal{R}$ is a set of mutually disjoint rays of $G$ such that $\mathcal{R}\subseteq \varepsilon$.
    \item[(b)] $\mathcal{P}$ is a set of $\mathcal{R}$-independent paths of $G$.
    \item[(c)] $|\mathcal{R}|=\kappa'$. 
    \item[(d)] The $(\mathcal{R},\mathcal{P})$-associated $(\mathcal{V},\mathcal{E})$  is connected. 
\end{itemize}
In particular, $HC(\kappa)$ fails.
\end{lemma}
    \begin{proof} Let $J$ be the $(T,X)$ be a $(\lambda,\kappa)$-graph given by Lemma \ref{Lemma_nosubgraph} , where $(T \cup X,\subseteq)$ is a sub-tree of $\lambda^{\leq \omega}$. Let $G$ be a sparse $(T\cup X)$-graph and let $H$ be the inflation of $G$.
     It follows by \cite{geschke2023halin} that has a sole end $\varepsilon$ such that $\deg(\varepsilon)=\kappa'$. Fixed $\kappa'$ a regular cardinal such that $\lambda < \kappa' \leq \kappa$, we will prove that there is no $(\mathcal{R},\mathcal{P})$ satisfying (a) to (d). 

     Towards a contradiction, suppose that there is $(R,P)$ that satisfies (a) to (d).  

     By Lemma \ref{lemma_stars} we may assume that $(\mathcal{R},\mathcal{P})$ contains $S$ a $\kappa''$-star of rays for some $\kappa'' > \lambda$ such that $\kappa'' \leq \kappa'$ and $\kappa''$ is a regular cardinal.  Denote by $r_0$ the center ray of $S$ and $\langle r_i \mid 0 < i < \kappa'' \rangle$ the leaf rays of $S$. 

     Let $\overline{A} =\{ s \in T \mid \exists n \exists v  \left(n \in \omega \wedge v \in T \cup X  \wedge (v,n)    \in r_0 \wedge s <_{T} v \right) \}$. The set $\overline{A}$ is countable and $|(T \setminus \overline{A})\times \mathbb{N}|=\lambda$. 

    Since the elements of $S$ are mutually disjoint and $|T\times \mathbb{N}|= \lambda < \kappa'$, then there are $\kappa''$ rays in $S$ which are  of the form $\{(v,n) \mid n \in \mathbb{N} \wedge n > k_0\}$ for some $v \in X$ and some $k_0 \in \mathbb{N}$. So we may assume that all leaf rays of $S$ are a tail end of some ray of the form $\{(v,n) \mid n \in \mathbb{N}\}$ for some $v \in X$. 

     Since $(T,X)$ has no $(\aleph_0,\kappa'')$-subgraph, it follows that $|\{v \in X \mid  |N_{J}(v) \cap \overline{A}| \text{ is infinite } \}| < \kappa''$. Then the set $|\{ v \in X \mid \exists k_0 (\{(v,n) \mid n \in \mathbb{N} \wedge n > k_0\} \in S) \wedge N_{J}(v) \cap \overline{A} \text{ is finite } \}|=\kappa''$, i.e. the set of $v \in X$ such that some tail end of $\{v\}\times\mathbb{N}$ is a ray of $S$ and $N_{J}(v)\cap \overline{A}$ is finite has cardinality $\kappa''$. 

     If $N_{J}(v)\cap \overline{A}$ is finite, and for some $k \in \omega$ the ray $\{(v,n) \mid n > k \}=r$ is a leaf ray of $S$, then  $N_{H}(r)\cap (\overline{A}\times \mathbb{N})$ is also finite. Therefore there are infinitely many paths $p \in \mathcal{P}$ connecting $r$ to $r_0$ such that $p \cap \left((T \setminus \overline{A})\times \mathbb{N})\right) \neq \emptyset$. 

     Thus there are $\kappa''$ $\mathcal{R}$-independent paths with internal points in $(T \setminus \overline{A})\times \mathbb{N})$, which is impossible since $|(T \setminus \overline{A})\times \mathbb{N})|=\lambda < \kappa''$. Hence there can be no such $(\mathcal{R},\mathcal{P})$ satisfying (a) to (d). \qedhere

    \end{proof}
In the next lemmas we prove a few basic facts about the $\resh$ function.
\begin{lemma} \label{lemma_reshprop1}
    Let $\kappa > \mathfrak{c}$, then one of the following hold: 
    \begin{itemize}
        \item[(a)] there is a $\gamma$ such that  $\kappa \in \bigl (\resh(\gamma),\resh(\gamma+1)\bigr ]$

        \item[(b)] there is a $\gamma$ such that  $\kappa =\resh(\gamma)$, $\gamma$ is a limit ordinal such that $\cf(\gamma) = \omega $ or $\cf(\gamma)$ is an $\omega$-inaccessible cardinal.
        \item[(c)] there is a $\gamma$ such that  $\kappa=\resh(\gamma)$, $\gamma$ is a limit ordinal such that $\cf(\gamma)\neq \omega$ and $\cf(\gamma)$ is not an $\omega$-inaccessible cardinal.
        
        \end{itemize}
\end{lemma}
\begin{proof}
    Since $\resh:\Ord \rightarrow \Ord$ is an increasing function with $\resh(0)=\mathfrak{c}$, it follows that for any given cardinal $\kappa$ that is strictly greater than $\mathfrak{c}$,  there is a least $\gamma$ such that $\kappa \in \resh(\gamma+1)$. We consider two cases: First we assume that $\gamma$ is a limit ordinal. Then $\resh(\gamma)=sup_{\delta \in \gamma} \resh(\delta)$ and therefore, for every $\alpha < \resh(\gamma)$ there is a $\delta < \gamma$ such that $\alpha < \resh(\delta+1)$. From the minimality of $\gamma+1$ it follows that $\resh(\gamma)\leq \kappa$. If $\gamma=\delta+1$, then $\resh(\delta+1) \leq \kappa$ from the minimality of $\gamma$. Thus $\kappa \in [\resh(\gamma),\resh(\gamma+1)]$. If $\resh(\gamma) < \kappa$, then $\kappa \in (\resh(\gamma),\resh(\gamma+1)]$ and we are in case (a). If $\kappa=\resh(\gamma)$, then we are in case (b) or in case (c).

    If $\gamma=\delta+1$ for some $\delta$, from the minimality of $\gamma$ we have $\resh(\delta) \leq \kappa$, then $\kappa \in [\resh(\delta),\resh(\delta+1)]$. \qedhere

\end{proof}

\begin{lemma} \label{lemma_reshprop2}
\begin{itemize}
    \item[(a)] For every $\gamma$ we have that $\resh(\gamma+1)^{\omega}=\resh(\gamma+1)$.
    \item[(b)] For every limit $\gamma$ we have that $\resh(\gamma)$ is an $\omega$-inaccessible cardinal.
    \item[(c)] If $\gamma$ is not a limit ordinal with $cf(\gamma)=\omega$, then $\resh(\gamma)^{+n}$ is an $\omega$-inaccessible cardinal for every $n \in (\omega \cup \{\omega\}) \setminus \{0\}$ and $\resh(\gamma)^{+\omega} < \resh(\gamma+1)$.
\end{itemize}
\begin{proof}
    \begin{itemize}
        \item[(a)] We have by definition that either $\resh(\gamma+1)=(\resh(\gamma)^{+\omega})^{\omega}$ or $\resh(\gamma+1)=(\resh(\gamma))^{\omega}$. Since $((\resh(\gamma)^{+\omega})^{\omega})^{\omega}=(\resh(\gamma)^{+\omega})^{\omega}$ and $(\resh(\gamma)^{\omega})^{\omega}=(\resh(\gamma))^{\omega}$ (a) follows.
        \item[(b)] Let $\gamma$ be a limit ordinal and let $\mu$ be a cardinal such that $\mu < \resh(\gamma) = \bigcup_{\alpha \in \gamma}\resh(\alpha)$. Then, there is $\alpha < \gamma$ such that $\mu < \resh(\alpha) < \resh(\alpha+1) < \resh(\gamma)$. Therefore $\mu^{\omega} \leq \resh(\alpha+1)^{\omega} = \resh(\alpha+1) < \resh(\gamma)$. Thus $\resh(\gamma)$ is $\omega$-inaccessible.
        \item[(c)] We start proving that for $\gamma$ that is not limit with $cf(\gamma)=\omega$, it follows that $\resh(\gamma)^{\omega}=\resh(\gamma)$. If $\gamma=\delta+1$ for some $\delta$, this follows from (a). If $\gamma$ is a limit, $\cf(\gamma)\geq \omega_{1}$, hence every $x \in \resh(\gamma)^{\omega}$ belongs to some $\alpha^{\omega}$ for $\alpha \in \resh(\gamma)$. From the definition of $\resh(\gamma)$, there is $\delta \in \gamma$ such that $\alpha < \resh(\delta+1) < \resh(\gamma)$. From (a) we have $\alpha^{\omega} \leq \resh(\delta+1)^{\omega}=\resh(\delta+1) < \resh(\gamma)$. Therefore $|\resh(\gamma)^{\omega}| \leq |\bigcup_{\alpha \in \resh(\gamma)} \alpha^{\omega}| \leq \resh(\gamma)\times \resh(\gamma)=\resh(\gamma)$.

       By Proposition \ref{Prop_HCSucc2} it follows that $\resh(\gamma)^{+n}$ is an $\omega$-inaccessible cardinal for all $n \in (\omega \cup \{ \omega\})\setminus\{0\}$.    \end{itemize} \qedhere

\end{proof}
    
\end{lemma}

\begin{notation}
    If  $\kappa$ is not an $\omega$-inaccessible cardinal, we denote by $\theta(\kappa)$ the least cardinal $\theta < \kappa$ such that $\theta^{\omega} > \kappa$.

\end{notation}

  \begin{rmk}
      Notice that if $\alpha < \theta(\kappa)$, since $(\alpha^{\omega})^{\omega}  = \alpha^{\omega}   < \kappa$, it follows that $\alpha^{\omega} < \theta(\kappa)$. Moreover, $\cf(\theta(\kappa)) = \omega$, as otherwise,  we have $\theta(\kappa)^{\omega}=\theta(\kappa)$ from the fact that $\theta(\kappa)$ is an $\omega$-inaccessible cardinal. 

  \end{rmk}

We need to state one last theorem before proving our main result: 

\begin{thm}\cite[Theorem 9.1]{geschke2023halin} \label{Theorem_InflationCE}
    If $\kappa$ is a cardinal such that $ \HC(\cf(\kappa))$ fails, then $\HC(\kappa)$ fails.
\end{thm}

\begin{thm} \label{main1}

   Suppose $\kappa > \mathfrak{c}$ is any cardinal and that if there is  a limit ordinal $\gamma$ such that $\resh(\gamma)=\kappa$,  then $cf(\gamma) > \mathfrak{c}$. Then the following hold:
    \begin{itemize}

                \item[(a)] If $\kappa \in \bigl (\resh(\gamma),\resh(\gamma)^{+\omega} \bigr ]$, for some $\gamma$ with $cf(\gamma)\neq\omega$, then $\HC(\kappa)$ holds.
                \item[(b)] If $\kappa \in \bigl (\resh(\gamma)^{+\omega},\resh(\gamma+1)\bigr ]$, for some $\gamma$ with $\cf(\gamma)\neq\omega$, then $\HC(\kappa)$ fails.
                              \item[(c)] If $\kappa \in (\resh(\gamma),\resh(\gamma+1)]$ for some $\gamma$ such that $cf(\gamma)=\omega$, then $\HC(\kappa)$ fails. 
        
        \item[(d)] If  $\kappa =\resh(\gamma)$ is a singular cardinal such that $\cf(\gamma) = \omega $ or $\cf(\gamma)$ is an $\omega$-inaccessible cardinal, then $\HC(\kappa)$ holds.
        \item[(e)] If $\kappa=\resh(\gamma)$ is a singular cardinal such that $\cf(\gamma)\neq \omega$ and $\cf(\gamma)$ is not an $\omega$-inaccessible cardinal, then $\HC(\kappa) $ fails.
        
        \end{itemize}
\end{thm}

\begin{proof}

    Case  (a): From (c) of Lemma \ref{lemma_reshprop2} it follows that $\kappa$ is $\omega$-inaccessible. If $\kappa \in \bigl ( \resh(\gamma),\resh(\gamma)^{+\omega} \bigr )$, it is a successor cardinal, then $\kappa$ is regular. We can apply Theorem \ref{thm_strong} to $\kappa$ and obtain $\HC(\kappa)$. If $\kappa=\resh(\gamma)^{+\omega}$, since it is a is $\omega$-inaccessible and $\cf(\resh(\gamma)^{+\omega})=\omega$, it follows from Lemma \ref{LemmaHCsingular} that $\HC(\resh(\gamma)^{+\omega})$ holds. 
    
    Case  (b): If we let $\lambda = \resh(\gamma)^{+\omega}$, then $\lambda$ is an $\omega$-inaccessible cardinal with $\cf(\lambda)=\omega$. By the definition of $\resh$, we have $\resh(\gamma+1)=\lambda^{+\omega}$. Therefore, by Lemma \ref{lemma_counterexample}, it follows that $\neg \HC(\kappa)$. 
    
    Case (c): Analogous to the previous case, let $\lambda =\resh(\gamma)$, then $\lambda$ is an $\omega$-inaccessible cardinal and $cf(\resh(\gamma))=\omega$. By Lemma \ref{lemma_counterexample}, it follows that $\neg \HC(\kappa)$. 

Case  (d): From (b) of Lemma \ref{lemma_reshprop2} we have that $\resh(\gamma)$ is an $\omega$-inaccessible cardinal and by hypothesis $\cf(\resh(\gamma))=\omega$ or $\cf(\resh(\gamma))$ is an $\omega$-inaccessible cardinal. In both cases we may apply Lemma \ref{LemmaHCsingular} and it follows that $\HC(\kappa)$ holds.
   
    Case (e): Let $\mu=\theta(cf(\kappa))$, then $\mu$ is an $\omega$-inaccessible  cardinal. Then $\cf(\kappa) \in \bigl ( \mu, \mu^{\omega}\bigr ]$.  From Lemma \ref{lemma_counterexample}  we have that $\HC(\cf(\kappa))$ fails. By Theorem \ref{Theorem_InflationCE}, we have that $\HC(\kappa)$ fails. \qedhere

\end{proof}


\section{ $\SCH$ and its relation to $\HC$ \label{Section_Merimovich}}

In this section we discuss what values  the function $\resh $ assumes in certain models of $\ZFC$. The relation between $\HC$ and $\SCH$ and we also state some questions.

\begin{rmk}
    Given $ \omega < \alpha < \omega_1 $, by results due to Magidor and Shelah \cite{GitikMagidor}, assuming large cardinals, it is consistent that \gch holds below $\aleph_{\omega}$ and $\aleph_{\omega}^{\omega}=\aleph_{\alpha+1}$. In this context we have $\resh(1)=\mathfrak{c} = \aleph_{1}$ and $\resh(2)=\aleph_{\omega}^{\omega}=\aleph_{\alpha+1}$. Applying Theorem \ref{main1} it follows that $\HC(\kappa)$ fails for every $\kappa \in \bigl ( \resh(1)^{+\omega},\resh(2)\bigr ]=(\aleph_{\omega},\aleph_{\omega}^{\omega}\bigr ]$ and $\HC(\kappa)$ holds for every $\kappa \in \bigl ( \resh(1),\resh(1)^{+\omega}\bigr ]=(\aleph_{1},\aleph_{\omega}\bigr ]$.
\end{rmk}

\begin{defn}
The \emph{Singular Cardinal Hypothesis above $\lambda$} states the following:
\begin{center}
$\SCH_{\geq \lambda }$: $ \mathfrak{c}  < \lambda$ and $\kappa^{\mathrm{cf}(\kappa)}=\kappa^{+}$ for every $\kappa > \lambda$ such that $2^{\mathrm{cf}(\kappa)} < \kappa$.    
\end{center}

\end{defn}

As a consequence of the next corollary, we will derive large cardinal strength from failures of $\HC$.
\begin{cor}\label{succSCH}
 If $\SCH_{\geq \lambda }$ holds, then $\HC(\kappa)$ holds for all successor cardinal $\kappa$ such that $ \kappa > \lambda$ and are not of the form $\mu^{+}$ for some $\mu $ such that $\mathrm{cf}(\mu)=\omega$.    
\end{cor}
\begin{proof}
    Suppose that we have cardinals $\mu$ and $\kappa$ such that $\mu^{+}=\kappa > \lambda  \geq (2 ^{\omega})^{+}$ and $\mathrm{cf}(\mu)\neq \omega$. Then $\mu > 2^{\omega}$ and $\cf(\mu)>\omega$ by \cite[Theorem 5.22 (ii)(b)]{Jechbook}, $\SCH$ above $\lambda$ implies that $\mu^{\omega}=\mu$. Therefore, by Proposition \ref{Prop_HCSucc2} it follows that $\kappa$ is a regular $\omega$-inaccessible cardinal. Then, by Theorem \ref{thm_strong}, $\HC(\mu^{+})$ holds.
\end{proof}
The following corollary is an immediate consequence of Gitik's landmark result \cite{GITIKSCH} on the consistency strength of the negation of $\SCH$.
\begin{cor} \label{largecardinals}
    Suppose that $\HC(\kappa^{+})$ fails for some $\kappa^{+}$ such that $\mathrm{cf}(\kappa) \neq \omega$ and $\kappa^{+}>\mathfrak{c}$. Then there is an inner model $\mathcal{M}$ with a measurable cardinal $\theta$ such that $o^{\mathcal{M}}(\theta)=(\theta^{++})^{\mathcal{M}}$.
\end{cor}

Note that if $\kappa$ is $\omega$-inaccessible, then $\HC(\kappa)$ is not affected by $\kappa^{\omega}$. On the other hand, in this case, $\kappa^{\omega}$ will affect $\HC(\kappa^{+n})$ depending on whether $\kappa^{\omega} \geq \kappa^{+n}$.


\begin{cor}\label{merimovich} Suppose there is a class of strong cardinals, then given $n \in \omega$ it is consistent that for every cardinal $\mu$ it holds that $\HC(\mu^{+\omega+k})$ fails for every $k \leq n$. 
\end{cor}

\begin{proof}  Assuming large cardinals by \cite{merimovichmodel} for every $n \in \omega$ it is consistent that $2^{\lambda}=\lambda^{+n}$ holds for every cardinal $\lambda$. Let $\lambda=\mu^{+\omega}$ for some cardinal $\mu$. Then $\mathrm{cf}(\lambda)=\omega$ and $\theta < \lambda$ implies $\theta^{\omega} \leq \theta^{\theta}=2^{\theta} = \theta^{+n} < \mu^{+\omega}=\lambda $ These inequalities show that $\lambda$ is $\omega$-inaccessible. Since $\lambda^{\omega}=\lambda^{+n}$, it follows from Lemma \ref{lemma_counterexample} applied to $\lambda$ that $\HC(\lambda^{+m})$ fails for every $m$ such that $1 \leq m \leq n$. \qedhere

\end{proof}

We conclude this section with a question: In Theorem \ref{main1} we were able to address the cardinals $\kappa \geq \mathfrak{c}$. So a natural question is what happens concerning $HC(\kappa)$ for those cardinals $\kappa < 2^{\omega}$. More specifically we ask: 


\begin{question}
    Suppose that $\kappa$ is a cardinal such that $\aleph_{\omega} < \kappa < 2^{\aleph_{0}} = \aleph_{\omega}^{\omega}$. Is it independent of $\ZFC$ whether $HC(\kappa)$ holds?
\end{question}

\section{Acknowledgments}

We are grateful to Matheus Koveroff Bellini and Leandro Aurichi for insightful and stimulating discussions on Halin’s end degree conjecture. We thank Michel Gaspar for his invaluable contributions to editing earlier versions of this manuscript. We thank the anonymous referee for numerous corrections that  significantly improved this manuscript. 

\providecommand{\bysame}{\leavevmode\hbox to3em{\hrulefill}\thinspace}
\providecommand{\MR}{\relax\ifhmode\unskip\space\fi MR }
\providecommand{\MRhref}[2]{%
  \href{http://www.ams.org/mathscinet-getitem?mr=#1}{#2}
}
\providecommand{\href}[2]{#2}


\begin{thebibliography}{10}

\bibitem{aurichifernandesjunior}
Leandro Aurichi, Gabriel Fernandes, and Paulo~Magalhães Júnior, \emph{On ends of degree $\omega_1$}, 2024.

\bibitem{diestel}
Reinhard Diestel, \emph{Graph theory, 6th edition}, Graduate texts in mathematics \textbf{173} (2025).

\bibitem{geschke2023halin}
Stefan Geschke, Jan Kurkofka, Ruben Melcher, and Max Pitz, \emph{Halin’s end degree conjecture}, Israel Journal of Mathematics \textbf{253} (2023), no.~2, 617--645.

\bibitem{GITIKSCH}
Moti Gitik, \emph{The strenght of the failure of the singular cardinal hypothesis}, Annals of Pure and Applied Logic \textbf{51} (1991), no.~3, 215--240.

\bibitem{GitikMagidor}
Moti Gitik and Menachem Magidor, \emph{The singular cardinal hypothesis revisited}, Set theory of the continuum, Springer, 1992, pp.~243--279.

\bibitem{halin1965maximalzahl}
Rudolf Halin, \emph{{\"U}ber die maximalzahl fremder unendlicher wege in graphen}, Mathematische Nachrichten \textbf{30} (1965), no.~1-2, 63--85.

\bibitem{halin2000miscellaneous}
\bysame, \emph{Miscellaneous problems on infinite graphs}, Journal of Graph Theory \textbf{35} (2000), no.~2, 128--151.

\bibitem{Jechbook}
Thomas Jech, \emph{Set theory}, millennium ed., Springer Monographs in Mathematics, Springer-Verlag, Berlin, 2003. \MR{1940513}

\bibitem{merimovichmodel}
Carmi Merimovich, \emph{A power function with a fixed finite gap everywhere}, The Journal of Symbolic Logic \textbf{72} (2007), no.~2, 361--417.

\bibitem{CardArthSkept}
Saharon Shelah, \emph{Cardinal arithmetic for skeptics}, American Mathematical Society \textbf{26} (1992), no.~2.

\end{thebibliography}
\end{document}